\documentclass{scrartcl}
\usepackage{lmodern}
\usepackage{amssymb}
\usepackage{amsmath}
\usepackage{amsthm}
\usepackage{xcolor}
\usepackage{bbm}
\usepackage{hyperref}
\usepackage[capitalise]{cleveref}
\usepackage{float}

\crefname{equation}{}{}
\crefname{enumi}{}{}

\newtheorem{theorem}{Theorem}
\newtheorem{lemma}{Lemma}
\newtheorem{corollary}{Corollary}
\theoremstyle{definition}
\newtheorem{remark}{Remark}

\newtheorem{example}{Example}

\newtheorem{problem}{Open Problem}

\allowdisplaybreaks

\newcommand{\F}{\mathbb{F}}
\newcommand{\N}{\mathbb{N}}
\newcommand{\Z}{\mathbb{Z}}
\newcommand{\R}{\mathcal{R}}
\newcommand{\M}{\mathcal{M}}
\newcommand{\mC}{\mathcal{C}}
\newcommand{\abs}[1]{\left|#1\right|}
\newcommand{\floor}[1]{\left\lfloor#1\right\rfloor}

\newcommand{\one}{\mathbbm 1}

\DeclareMathOperator{\id}{\operatorname{id}}
\DeclareMathOperator{\ord}{\operatorname{ord}}

\DeclareMathOperator{\spn}{\operatorname{span}}

\DeclareMathOperator{\inv}{\operatorname{inv}}

\title{There are siblings of \texorpdfstring{$\chi$}{chi} which are permutations for \texorpdfstring{$n$}{n} even}

\author{Björn Kriepke and Gohar Kyureghyan}

\begin{document}
\maketitle

\begin{abstract}
Let $\one$  be the all-one vector and
$\odot$ denote the component-wise multiplication of two vectors in $\F_2^n$. We study the vector space
$\Gamma_n$ over $\F_2$ generated by the functions $\gamma_{2k}:\F_2^n \to \F_2^n, k\geq 0$,
where
$$
\gamma_{2k} = S^{2k}\odot(\one+S^{2k-1})\odot(\one+S^{2k-3})\odot\ldots\odot(\one+S)
$$
and
 $S:\F_2^n\to\F_2^n$ is the cyclic left shift function. The functions in $\Gamma_n$ are shift-invariant and the well known $\chi$ function used
 in several cryptographic primitives is contained in
 $\Gamma_n$. 
 For even $n$, we show that the 
 permutations from $\Gamma_n$ with respect to composition form an Abelian group, which is isomorphic to the
 unit group of the residue ring $\F_2[X]/(X^n +X^{n/2})$. This isomorphism yields an efficient theoretic and 
 algorithmic method for constructing and studying a rich family of shift-invariant permutations on $\F_2^n$
 which are natural generalizations of $\chi$.  
 To demonstrate it, we apply the obtained results
 to investigate the function $\gamma_0 +\gamma_2+\gamma_4$ on $\F_2^n$.

\end{abstract}

\section{Introduction}

Let $\F_2$ be the field with two elements. The function $\chi:\F_2^n\to \F_2^n$ mapping $x$ to $\chi(x)$ is defined by 
\begin{equation*}
    \chi(x)_i = x_i + x_{i+2}(1+x_{i+1}),
\end{equation*}
where  $\chi(x)_i$ denotes the $i$th coordinate function of $\chi$ 
for $1\leq i \leq n$. The computations of indexes are done modulo $n$. 
The function $\chi$ is used in several important cryptographic primitives, for example SHA-3 \cite{nistSHA3StandardPermutationBased2015} and Ascon \cite{dobraunig2021ascon}. Mathematical properties of $\chi$ are studied in \cite{graner2025bijectivity,liu2022inverse,liu2025finding,schoone2024algebraic,schoone2024state}.  The function $\chi$ commutes with the left and right shift functions. Such functions are called shift invariant and can be effectively
implemented in practice. Shift invariant functions are closely related with cellular automata (CA) \cite{kari2005theory}, they are one-dimensional CA on periodic configurations. \\

The shift invariant permutation on $\F_2^n$
are of particular interest in cryptological applications.
It is well-known that $\chi$ is a permutation if and only if $n$ is odd (see for example \cite{daemen1995cipher}). However, in many
of designs of block ciphers, the block length is even and 
therefore constructions of permutations with good cryptographic properties required for $n$ even.
In \cite{leo-eurocrypt-2025-35126}, a permutation for $n=2m$ with $m$ even is constructed via concatenation of two $\chi$ functions and the addition of a suitable linear function. This method was extended to cover any even $n$ in a recent paper \cite{andreoli2025generalizations}. However, the resulting permutations are not shift-invariant. \\

In our recent paper \cite{kriepke2024algebraic}, it is observed that for an odd $n$ a number of properties of $\chi$
can be explained  by studying
the binary polynomial $1+X$ modulo $X^{(n+1)/2}$. 
For example, its order, its cycle structure, the explicit form of its iterates and its inverse function can be effectively derived  via this correspondence. 
We outline briefly the main ideas  and results of \cite{kriepke2024algebraic} below. \\

Let $\one$  denote the vector $(1,\ldots,1)\in\F_2^n$ and  $S:\F_2^n\to\F_2^n$ be the cyclic left shift function, i.e., 
$$S(x_1,\ldots, x_n)=(x_2,\ldots, x_n, x_1).$$
With the symbol $\odot$ we denote the component-wise multiplication of two vectors $x,y\in\F_2^n$. More precisely, $z=x\odot y$ denotes the vector with the $i$th component $z_i=x_i\cdot y_i$ for all $1\leq i \leq n$. Note that $S$ is linear. Furthermore, $S(x\odot y)=S(x)\odot S(y)$. The operation $\odot$ is commutative and distributive with respect to the addition, that is, $x\odot y=y\odot x$ and $x\odot(y+z)=x\odot y+x\odot z$.
Define $\gamma_0=\id$ and $\gamma_{2k}:\F_2^n\to\F_2^n$ for $k\geq 1$ by 
\begin{equation*}
    \gamma_{2k} := S^{2k}\odot(\one+S^{2k-1})\odot(\one+S^{2k-3})\odot\ldots\odot(\one+S).
\end{equation*}
The function $\gamma_{2k}$ is shift-invariant, since
$\gamma_{2k} \circ S = S\circ \gamma_{2k}$. Note that $\chi = \gamma_0 + \gamma_2$.
\\

For $n$ odd, it is shown in \cite{kriepke2024algebraic}, that the iterates of the function $\chi$
are linear combinations of functions $\gamma_{2k}$ over $\F_2$. This motivated the study of the linear span $\Gamma_n $ of 
the functions $\gamma_{2k}$ over $\F_2$. If $n$ is odd,  then $\Gamma_n$ has dimension $(n+1)/2$, since $\gamma_{2k}=0$ for every $k$ with  $2k > n$ and
the functions  $\gamma_0, \gamma_2, \gamma_4,\ldots, \gamma_{n-1}$ are linearly independent 
over $\F_2$. 
Since $\gamma_{2k}(0)=\gamma_{2k}(\one)=0$ for every $k\geq 1$, it follows that every permutation in $\Gamma_n$ is contained in $G_n=\gamma_0+\spn\{\gamma_2,\ldots,\gamma_{n-1}\}$. 
Notably, it was shown in \cite{kriepke2024algebraic} that $G_n$ is a monoid with respect to composition, which is isomorphic to 
\begin{equation*}
    \M_n = \left\{\left[1+\sum_{i=1}^ta_i X^i\right]\right\}\subseteq \F_2[X]/(X^{(n+1)/2}),
\end{equation*}
that is  to the unit group of the residue ring $\F_2[X]/(X^{(n+1)/2})$. From the fact that $(\M_n,\cdot)$ is an Abelian group it follows that $(G_n,\circ)$ is also an Abelian group. Hence, $G_n$ consists of permutations which commute. 
This isomorphism allows to construct, both theoretically and algorithmically, a rich
family of shift-invariant permutations. Further, it reduces
the study of specific properties of these permutations to that of the corresponding binary polynomials. 
To the best of our knowledge, the previously known families of shift-invariant
permutations are sporadic, in the sense that they are described by a specific algebraic formula for suitable 
infinite set of $n$. \\

The ideas of \cite{kriepke2024algebraic} are adapted in \cite{lyu2025generalized} to construct $\chi$-like permutations which are called $\chi_{n,m}$. If $m=2$ then $\chi_{n,2}$ coincides with $\chi$. Using similar steps as \cite{kriepke2024algebraic}, they show that $\chi_{n,m}$ is contained in an Abelian group which is isomorphic to the unit group of $\F_2[X]/(X^{\floor{n/m}+1})$ if $m$ does not divide $n$. \\

In this paper we generalize the ideas of \cite{kriepke2024algebraic} to study the linear span $\Gamma_n$ of the functions $\gamma_{2k}$ when $n$ is even. We show that $\Gamma_n$ is an $n$-dimensional vector space. 
Similarly to the $n$ odd case, all permutations in $\Gamma_n$ are contained in $G_n=\gamma_0+\spn\{\gamma_{2i}: i\geq 1\}$. 
However, in contrast to the $n$ odd case, $G_n$ also contains non-bijective functions if $n$ is even. 
We show that the subset of bijections in $\Gamma_n$ is an Abelian group which is isomorphic to the unit group of the ring $\F_2[X]/(X^n+X^{n/2})$. 
We expect that analogously to the case $n$ odd, the ideas of this paper can be adapted to cover the remaining case $m\mid n$ for the
functions studied in \cite{lyu2025generalized}. \\

The study of units in $\F_2[X]/(X^n+ X^{n/2})$ is closely related to the factorization 
of the polynomial $X^{n/2}+1$ over $\F_2$, which is a
challenging classical problem in algebra and number theory.
This explains why the behavior of permutation from $\Gamma_n$ is not as consistent as for $n$ odd. 
For $n$ odd the corresponding residue ring is modulo $X^{(n+1)/2}$, which has the unique irreducible factor $X$, independently of $n$. 
For $n$ even, the factorization of $X^{n/2} +1$ depends on $n$ and in general is not easy to write down \cite{graner2024irreducible}.
In spite of this, similarly to the $n$ odd case, the isomorphism of the set of permutations in $\Gamma_n$
with the unit group of $\F_2[X]/(X^n+ X^{n/2})$
yields an effective and powerful method for
constructing and studying a rich family of shift-invariant 
permutations, which are natural generalizations of $\chi$. 
We demonstrate our methods by applying them to study the map $\kappa=\gamma_0+\gamma_2+\gamma_4$ which is the simplest function in $G_n$ which can be a permutation for $n$ even. \\

\section{\texorpdfstring{$\chi$}{chi}-like shift-invariant permutations for \texorpdfstring{$n$}{n} even}\label{sect:compositions_gammas}

 The next example demonstrates  that the  behavior of the set of functions $\gamma_{2k}$ on $\F_2^n$  depends strongly on the parity of $n$.

\begin{example}
    First consider an odd $n$. We take $n=7$. Then the first few functions $\gamma_{2k}$ are given by
    \begin{align*}
        \gamma_0    &= \id \\
        \gamma_2    &=  S^2    \odot (\one + S) \\
        \gamma_4    &= S^4    \odot (\one + S^3) \odot (\one + S) \\
        \gamma_6    &= S^6    \odot (\one + S^5) \odot (\one + S^3) \odot (\one + S) \\
        \gamma_8    &= S^8    \odot (\one + S^7) \odot (\one + S^5) \odot (\one + S^3) \odot (\one + S).
    \end{align*}
    Since $S^8=S^1$ we get the product $S^1\odot (\one+S)=0$ in the last line, implying that $\gamma_8=0$. Moreover, all functions $\gamma_0,\ldots,\gamma_6$ have distinct algebraic degree. Recall,
    that the algebraic degree of a function $g:\F_2^n \to \F_2^n$ is defined as the maximum over the multivariate degrees of 
    its component functions of $g(x)_i, 1\leq i \leq n$. 
    In \cite{kriepke2024algebraic} it is shown,  that $\gamma_{2k}=0$ for any $2k\geq n+1$ and that $\gamma_{2k}$ has algebraic degree $k+1$ for $2k<n+1$.

    For $n$ even this is not the case. Let $n=6$. Then the first  functions $\gamma_{2k}$ are:
    \begin{align*}
        \gamma_0    &= S^0=\id \\
        \gamma_2    &= S^2    \odot (\one + S) \\
        \gamma_4    &= S^4    \odot (\one + S^3) \odot (\one + S) \\
        \gamma_6    &= S^6    \odot (\one + S^5) \odot (\one + S^3) \odot (\one + S) \\
        \gamma_8    &= S^8    \odot (\one + S^7) \odot (\one + S^5) \odot (\one + S^3) \odot (\one + S)  \\
        \gamma_{10} &= S^{10} \odot (\one + S^9) \odot (\one + S^7) \odot (\one + S^5) \odot (\one + S^3) \odot (\one + S)  \\
        \gamma_{12} &= S^{12} \odot (\one + S^{11})\odot (\one + S^9) \odot (\one + S^7) \odot (\one + S^5) \odot (\one + S^3)
				\odot (\one + S)\\
				\gamma_{14} &= S^{14} \odot (\one + S^{13})\odot (\one + S^{11})\odot (\one + S^9) \odot (\one + S^7) \odot (\one + S^5) \odot (\one + S^3)
				\odot (\one + S)
    \end{align*}
    Using $S^6=\id =S^0$ and that $y\odot y=y$ for all $y\in\F_2^n$, we get 
    \begin{align*}
		     \gamma_6    &= S^0    \odot (\one + S^5) \odot (\one + S^3) \odot (\one + S) \\
        \gamma_8    &= S^2 \odot (\one + S^5) \odot (\one + S^3) \odot (\one + S) \\
        \gamma_{10} &= S^4 \odot (\one + S^5) \odot (\one + S^3) \odot (\one + S) \\
        \gamma_{12} &= S^0 \odot (\one + S^5) \odot (\one + S^3) \odot (\one + S) = \gamma_6\\
				\gamma_{14} &= S^2 \odot (\one + S^5) \odot (\one + S^3) \odot (\one + S) = \gamma_8
    \end{align*}
    In contrast to the situation for $n$ odd, we observe that $\gamma_{12}=\gamma_6$ and $\gamma_{14} =\gamma_{8}$. In particular, the functions eventually repeat instead of vanishing for large $k$. Furthermore their algebraic degrees are not distinct for $k\geq n/2$.
\end{example}

The above example suggests that for $n$ even and  $k$ large,  the functions $\gamma_{2k}$ are equal
to ones with smaller $k$. The next lemma confirms this.
\begin{lemma}\label{lem:gamma_2k_equals_gamma_small_new}
    Let $n$ be even and $\gamma_{2k}:\F_2^n \to \F_2^n$. If $k \geq n/2$, then 
		$$
		\gamma_{2k} = S^{{2k\bmod n}}\odot(\one + S^{n-1})\odot (\one + S^{n-3})\odot \ldots \odot (\one + S).
		$$
		 In particular,
        \begin{itemize}
            \item[(a)] if  $k, k' \geq n/2$ and $2k \equiv 2k' \bmod n$, then $\gamma_{2k}=\gamma_{2k'}$;
            \item[(b)] if $2k \geq n$ then $\gamma_{2k}=\gamma_{2(\frac{n}{2} + k\bmod \frac{n}{2})}$.
        \end{itemize}
\end{lemma}
\begin{proof}
The statement follows from the observations that on $\F_2^n$ it holds $S^n=\id$ and 
\begin{equation*}
    (\one +S^j)\odot\ldots\odot (\one+S^j)=\one+S^j.
\end{equation*}
\end{proof}
As a direct consequence of \cref{lem:gamma_2k_equals_gamma_small_new} we get
\begin{lemma}\label{lem:gamma_2k_equals_gamma_2kminusn}
    Let $n$ be even. Then $\gamma_{2k} = \gamma_{2k-n}$ on $\F_2^n$ if $k \geq n$. 
\end{lemma}

\begin{proof} 
    This follows from \cref{lem:gamma_2k_equals_gamma_small_new}, since $2k \equiv 2k-n \bmod n$ and $k, (k-n/2) \geq n/2$.
\end{proof}

 For even $n$, \cref{lem:gamma_2k_equals_gamma_small_new} implies that the set of functions $\gamma_{2k}:\F_2^n \to \F_2^n$ with $ k \geq 0$   reduces to that of  $\gamma_{2k}$ with $0 \leq k \leq n-1$. 
 Next we prove that these $n$ functions are linearly independent over $\F_2$.

\begin{lemma}\label{lem:gamma_2k_are_lin_independent}
 Let $n$ be even.   The functions $\gamma_0, \gamma_2, \gamma_4,\ldots, \gamma_{2n-2}$ of $\F_2^n$ are linearly independent over $\F_2$.
\end{lemma}

\begin{proof}
    Let $x=(x_0,x_1,\ldots,x_{n-1})$. Then the coordinate functions $\gamma_{2k}(x)_0$  are given by
    \begin{align*}
        \gamma_0(x)_0       &= x_0 \\
        \gamma_2(x)_0       &= x_2 (1+x_1) \\
        \gamma_4(x)_0       &= x_4 (1+x_3)(1+x_1) \\
        \vdots \\
        \gamma_{n-2}(x)_0   &= x_{n-2}(1+x_{n-3})\ldots (1+x_1) \\
        \gamma_n(x)_0       &= x_0 (1+x_{n-1})(1+x_{n-3})\ldots (1+x_1) \\
        \gamma_{n+2}(x)_0   &= x_2 (1+x_{n-1})(1+x_{n-3})\ldots (1+x_1) \\
        \vdots \\
        \gamma_{2n-2}(x)_0  &= x_{n-2}(1+x_{n-1})(1+x_{n-3})\ldots (1+x_1).
    \end{align*}
		Note that the highest degree monomials in these coordinate functions are pairwise different, implying that
		they are linearly independent over $\F_2$. 
\end{proof}

\begin{lemma}\label{lem:gamma_2k_algebraic_degree}
    Let $n$ be even, $k\geq 0$ and $d$ be the algebraic degree of $\gamma_{2k}:\F_2^n\to \F_2^n$. Then
    \begin{equation*}
        d = \begin{cases}
            k+1 & k\leq n/2 \\
            n/2+1 & k \geq n/2.
        \end{cases}
    \end{equation*}
\end{lemma}

\begin{proof}
    The statement follows immediately from the proof of \cref{lem:gamma_2k_are_lin_independent}, since all coordinate
		functions of $\gamma_{2k}$ have the same multivariate degree due to the shift-invariance.
\end{proof}

We consider the set of the linear combinations of the functions $\gamma_{2k}:\F_2^n\to \F_2^n$ over $\F_2$
\begin{equation*}
    \Gamma_n :=\left\{\sum_{k=0}^\ell a_k \gamma_{2k}: \ell \geq 0, ~a_k\in\F_2\right\} = \left\{\sum_{k=0}^{n-1} a_k \gamma_{2k}: a_k\in\F_2\right\}.
\end{equation*}
 Note that the latter equality follows from \cref{lem:gamma_2k_equals_gamma_small_new}. 

\begin{lemma}
 Let $n$ be even. Then the set   $\Gamma_n$ is an $n$-dimensional vector space over $\F_2$  and $\{\gamma_0,\gamma_2,\ldots,\gamma_{2n-2}\}$ is a basis of it. 
\end{lemma}

\begin{proof}
    Let $V$ be the $\F_2$-vector space of all functions of $\F_2^n$.
    The set $\Gamma_n$ is defined as the subspace of $V$ generated by functions $\gamma_{2k}$, $k\geq 0$. \cref{lem:gamma_2k_are_lin_independent} shows that $\{\gamma_0,\gamma_2,\ldots,\gamma_{n-1}\}$ is a basis of $\Gamma_n$. 
\end{proof}

Lemmas 4--7 of \cite{kriepke2024algebraic} show
  that compositions of the linear combinations of
 functions $\gamma_{2k}$ have a strongly structured behavior. However these lemmas
 hold  independently on the parity of $n$ (and in fact also for $\gamma_{2k}$ considered as functions on $\F_2^\N$ or $\F_2^\Z$), which makes the study of functions from 
 $\Gamma_n$ interesting also for $n$ even.
 Lemma 7 from those results is the key lemma for our considerations here, therefore  we state  it below. 
 Let $G_n$ to denote the coset of $\gamma_0 = \id$ with respect to the subspace
 generated by the remaining $\gamma_{2k}:\F_2^n \to \F_2^n, k\geq 1$, that is 
\begin{equation*}
    G_n:=\gamma_0+\spn\{\gamma_2, \gamma_4, \ldots, \gamma_{2n-2}\}. 
\end{equation*}
The next lemma shows that the left composition of 
functions from $G_n$ with any $\gamma_{2k}$ is again a 
linear combination from $\Gamma_n$.
\begin{lemma}[\cite{kriepke2024algebraic}, Lemma 7]\label{key_lem:first_closure_property}
    Let $m\geq 2$ be even and ${f=\gamma_0+\sum_{i=1}^t a_i \gamma_{2i} \in G_n}$. Then 
    \begin{equation*}
        \gamma_m\circ f = \sum_{i=0}^t a_i \gamma_{2i+m} \in \Gamma_n.
    \end{equation*}
\end{lemma}

\cref{key_lem:first_closure_property} indicates that the composition of the functions in $G_n$ has strong similarities with the multiplication
of polynomials over $\F_2$. Indeed, 
\begin{equation*}
    \gamma_2 \circ(\gamma_0+\gamma_4+\gamma_6) = \gamma_2+\gamma_6+\gamma_8
\end{equation*}
looks very much like
\begin{equation*}
    X^2 \cdot (1+X^4+X^6) = X^2+X^6+X^8.
\end{equation*}
After dividing the exponents of the polynomials by $2$, we then obtain correspondence of the form
\begin{equation*}
    \sum_{i=0}^t a_i \gamma_{2i} \mapsto \sum_{i=0}^t a_i X^i. 
\end{equation*}
Further, by \cref{lem:gamma_2k_equals_gamma_2kminusn} the equality $\gamma_{2k}=\gamma_{2k-n}$ holds  for all $k$ with $k\geq n$, when $n$ is even. This means that in the polynomial setting we need to require $X^k = X^{k-n/2}$, or equivalently $X^k+X^{k-n/2}=0$ if $k\geq n$. We can achieve this by considering the polynomials modulo $X^n+X^{n/2}$. In other words, we work in the residue
ring 
$$\mathcal{R}_n =\F_2[X]/(X^n+X^{n/2}),$$ 
where
 $(X^n+X^{n/2})$ is the ideal generated by the polynomial  $X^n+X^{n/2}$ in $\F_2[X]$.  \\

For a polynomial $f \in \F_2[X]$, let $[f]=f+(X^n+X^{n/2})$  denote the coset of $f$ in $\mathcal{R}_n$.
Then the maps $\varphi:\Gamma_n\to \R_n$ 
and $\psi: \R_n \to \Gamma_n$ are well-defined and inverse each to other, where 
\begin{equation*}
    \varphi\left(\sum_{i=0}^t a_i \gamma_{2i}\right)=\left[\sum_{i=0}^t a_i X^i\right]
\end{equation*}
and
\begin{equation*}
    \psi\left(\left[\sum_{i=0}^t a_i X^i\right]\right) =
    \sum_{i=0}^t a_i \gamma_{2i}.
\end{equation*}
In particular, $\varphi$ and $\psi$ are bijections.

 Further, we define the set
\begin{equation*}
    \M_n := \left\{\left[1+\sum_{i=1}^ta_i X^i\right] : t \geq 1, a_i \in \F_2 \right\}\subseteq \R_n.
\end{equation*}
Note that $\M_n$ is closed under the multiplication of $\R_n$ and  $[1]\in \M_n$. Hence $\M_n$ is a monoid with respect to the multiplication. Recall, that a monoid $(M,\star)$ is a set together with an associative operation $\star:M\times M\to M$ such that there exists a neutral element $e\in M$ with respect to $\star$. The invertible elements of $M$ are called units and their set is denoted by $M^*$.
Now we are ready to prove our first main result. It is worth to note that its  proof is an analogue of that for $n$ odd from \cite{kriepke2024algebraic}.

\begin{theorem}\label{thm:isom-monoid}
    Let $n\geq 2$ be even. Then
    \begin{enumerate}
        \item[(a)] $G_n$ is an Abelian monoid with respect to composition. The map $\varphi:G_n\to \M_n$ is a monoid isomorphism. In particular,  $G_n\cong \M_n$. 
        \item[(b)] The set  $G_n^*$ of invertible (equivalently, bijective) functions from $G_n$ is an Abelian group with respect to composition, which is isomorphic to the unit group $\R_n^* = \M_n^*$ of $\R_n$.
    \end{enumerate}
\end{theorem}

\begin{proof}
    We first show that $G_n$ is a monoid. Clearly $\id=\gamma_0\in G_n$, so we only need to show that $G_n$ is closed under composition. Set $\mC_n :=\spn\{\gamma_2,\ldots,\gamma_{2n-2}\}$, and then $G_n=\gamma_0+\mC_n$. By \cref{key_lem:first_closure_property}, we have $\gamma_m\circ g\in \mC_n$ for all $g\in G_n$. Let now $f,g\in G_n$ with $f=\gamma_0+\sum_{i=1}^k a_i\gamma_{2i}$. Then
    \begin{equation*}
        f\circ g = \left(\gamma_0 + \sum_{i=1}^k a_i \gamma_{2i}\right)\circ g 
            = \underbrace{\gamma_0\circ g}_{=g} + \left(\sum_{i=1}^k a_i \gamma_{2i}\right)\circ g 
            = g + \sum_{i=1}^k a_i \underbrace{\gamma_{2i}\circ g}_{\in \mathcal{C}_n} \in \gamma_0 + \mathcal{C}_n = G_n 
    \end{equation*}
    since $g\in\gamma_0+\mC_n$ and $\mC_n$ is a subspace.

    Next we show that $\varphi$ is a monoid homomorphism. From the definition it is clear that $\varphi$ (as a map $\Gamma_n\to \R_n$) is additive. Let $f=\sum_{i=0}^k a_i \gamma_{2i}, g=\sum_{j=0}^m b_j\gamma_{2j}\in G_n$ with $a_0=b_0=1$. Using \cref{key_lem:first_closure_property} we have 
    \begin{align*}
        \varphi(f\circ g) &= \varphi\left(\left(\sum_{i=0}^k a_i\gamma_{2i}\right)\circ g\right) 
                = \sum_{i=0}^k a_i \varphi(\gamma_{2i}\circ g) \\
                &= \sum_{i=0}^k a_i \varphi\left(\sum_{j=0}^m b_j \gamma_{2j+2i}\right) 
                = \sum_{i=0}^k a_i \left[\sum_{j=0}^m b_j X^{j+i}\right] \\
                &= \sum_{i=0}^k a_i \left[X^i\right]\left[\sum_{j=0}^m b_j X^j\right] 
                = \left[\sum_{i=0}^k a_i X^i\right]\left[\sum_{j=0}^m b_j X^j\right] 
                = \varphi(f)\varphi(g).
    \end{align*}
    Since $\varphi(\gamma_0)=[1]$ is the identity in $\M_n$, we get that $\varphi$ is a monoid homomorphism between $G_n$ and $\M_n$. As $\varphi$ is bijective it is then a monoid isomorphism. To complete the proof of part (a), note that $\M_n$ is an Abelian monoid and thus $G_n$ is Abelian too.  

    Observe that since $G_n$ is a finite monoid with respect to composition, a function $f\in G_n$ is bijective if and only if it is invertible in $G_n$. That is, the inverse function of a bijection from $G_n$ is contained in $G_n$ too. Since $\varphi(\id) = [1]$, the invertible functions in $G_n$ correspond to those in $\M_n$. It remains to observe that the units in $\R_n$ are contained in $\M_n$.
\end{proof}

Recall that for any $p(X) \in \F_2[X]$, the unit group of the residue ring $S=\F_2[X]/(p(X))$ is given by 
\begin{equation*}
    S^* = \{[g(X)]: g(X) \in \F_2[X] \text{ s.t. } \gcd(g(X),p(X))=1\}.
\end{equation*}
In particular, the unit group $S^*$ depends heavily on the factorization of $p(X)$ in $\F_2[X]$. If $[g(X)]\in \F_2[X]/(p(X))$ is a unit, then its inverse $[g(X)]^{-1}$ can be computed using the Extended Euclidean Algorithm in $\F_2[X]$. \\

The above discussions with \cref{thm:isom-monoid} yield our next main result.\\

\begin{theorem}\label{thm:permutation_iff_gcd_1}
    Let $n$ be even and $\displaystyle f=\gamma_0+\sum_{i=1}^{n-1} a_i \gamma_{2i}\in G_n$. Then $f$ is a permutation on $\F_2^n$  if and only if $\displaystyle{\varphi(f)=\left[1+\sum_{i=1}^{n-1} a_i X^i\right]\in \M_n}$ is a unit in $\F_2[X]/(X^n+X^{n/2})$, or equivalently, if and only if in $\F_2[X]$ it holds
    \begin{equation*}
        \gcd\left(1+\sum_{i=1}^{n-1} a_i X^i, 1+X^{n/2}\right)=\gcd\left(1+\sum_{i=1}^{n-1} a_i X^i, 1+X^{m}\right) =1,
    \end{equation*}
    where $m$ is the largest odd divisor of $n$.
\end{theorem}
\begin{proof}
  Let $n=2^s\cdot m$ with $s\geq 1$ and $m\geq 1$ odd.  Then 
    \begin{equation*}
        X^n + X^{n/2} = X^{n/2}(X^{n/2}+1) = X^{n/2}(X^m+1)^{2^{s-1}}
    \end{equation*}  
 and the statement follows from the discussions preceding the theorem.
\end{proof}

For $n$ odd it was shown in \cite{kriepke2024algebraic} that $G_n = \gamma_0 + \spn\{\gamma_2,\ldots,\gamma_{n-1}\}$ is isomorphic to the monoid
\begin{equation*}
    \M_n = \left\{\left[1+\sum_{i=1}^ta_i X^i\right]\right\}\subseteq \F_2[X]/(X^{(n+1)/2}),
\end{equation*}
which is the unit group of $\F_2[X]/(X^{(n+1)/2})$. This implies that $G_n$ itself is a group. However, for $n$ even, the unit group of $\R_n= \F_2[X]/(X^n+X^{n/2})$ is a proper subset of $\M_n$, as \cref{thm:permutation_iff_gcd_1} shows. 
As an example, it is well-known and easy to see that $\chi=\gamma_0+\gamma_2\in G_n$ is not a permutation on $\F_2^n$ if $n$ is even. \cref{thm:permutation_iff_gcd_1} gives an alternative proof for it: It holds that $\varphi(\chi)=[1+X]\in\F_2[X]/(X^n+X^{n/2})$ and  $\gcd(1+X, 1+X^{n/2})=1+X\neq 1$, implying that $\chi$ is not a permutation. The same argument shows that if a linear combination $f$ of $\gamma_{2k}$ defines a permutation on $\F_2^n$ for an even $n$, then $f$ needs to contain an odd number of non-zero terms.

It is worth to note that when $n$ is a power of $2$, then the converse of this fact is true as well, as stated in the next result.

\begin{corollary}\label{cor:power2}
    Let $n$ be a power of $2$. Then $f\in G_n$ is a permutation on $\F_2^n$ if and only if $f$ contains an odd number of non-zero terms.
\end{corollary}  

\begin{proof}
    Let $n=2^s$ and $\varphi(f)=[F(X)] \in \M_n$.
    By \cref{thm:permutation_iff_gcd_1}, $f$ is a permutation if and only if
    $\gcd(F(X),1+X)=1$ if and only if $F(1)=1$. The latter condition is equivalent to $F$, and hence $f$, having an odd number of non-zero terms.
\end{proof}

The above results show that the bijectivity of functions from $G_n$ depends on the factorization of $1+X^m$ over $\F_2$ which is a classical research topic in the field  and number theory. For a comprehensive survey and recent developments  on the factorization of $1+X^m$ over finite fields we refer to \cite{graner2024irreducible}. 

We would like to conclude this section by noting that the  
Chinese Remainder Theorem can be used 
to study the residue ring  $\F_2[X]/(X^n+X^{n/2})$. Indeed,
if $n=2^sm$, $m$ is odd, then $1+X^m=(1+X)q(X)$ with $q(X) \in \F_2[X] $ and $q(1) \ne 0$. Then by the Chinese Remainder Theorem,
\begin{equation*}
    \F_2[X]/(X^n+X^{n/2}) \cong \F_2[X]/(X^{n/2})\times \F_2[X]/((1+X)^{2^s})\times \F_2[X]/(q(X)^{2^s}).
\end{equation*}
 In particular the unit group of $\F_2[X]/(X^n+X^{n/2})$ is a product of the unit groups of the rings on the right-hand side, which are easier to study. 

\section{A closer look on the permutations from $G_n$}

Any shift-invariant function can be considered as a map $\F_2^n\to\F_2^n$ for every $n\geq 1$, since it is uniquely defined by its first coordinate function. In \cref{sect:compositions_gammas} we considered functions defined by linear combinations of $\gamma_{2k}, k\geq 0$ when $n$ is fixed. In this section we study the functions defined by a fixed linear combination of $\gamma_{2k}$ on $\F_2^n$ where $n$ runs over $\N$. For this purpose it is convenient to introduce the set $G=\gamma_0+\spn\{\gamma_2,\gamma_4,\ldots\}$ which is independent of $n$.

Following \cite{omland2022perm}, for a given shift-invariant $f$ we consider the set
\begin{equation*}
    \inv(f) = \{n\in\N: f \text{ is a permutation on } \F_2^n\}.
\end{equation*}
The next statement is well-known, c.f. \cite[Proposition 6.1]{daemen1995cipher} and the discussion afterwards. We give a short alternative proof.
\begin{lemma}\label{lem:inv_closed_divisors}
    Let $n,d$ be positive integers with $d\mid n$ and let $f:\F_2^n\to \F_2^n$ be a shift-invariant permutation. Then $f:\F_2^d\to \F_2^d$ is also a permutation. 
\end{lemma}

\begin{proof}
    Consider the set 
    \begin{equation*}
        H = \{(x, x,\ldots, x): x\in \F_2^d\}\subseteq \F_2^n.
    \end{equation*}
    As $f$ is a permutation on $\F_2^n$, then $\abs{f(H)}=\abs{H}=2^d$. On the other hand, 
    \begin{equation*}
        f(H) = \{(f(x), f(x), \ldots, f(x)): x\in\F_2^d\}
    \end{equation*}
    by shift-invariance and hence $\abs{f(H)} = \abs{\{f(x): x\in\F_2^d\}}$. The claim follows.
\end{proof}

\cref{lem:inv_closed_divisors} implies that $\inv(f)$ is closed under taking divisors. Equivalently, $\N\setminus\inv(f)$ is closed under taking multiples, which implies that there is a generating set $\xi(f)\subseteq\N\setminus\inv(f)$ of elements, so that $n\in\N\setminus\inv(f)$ if and only if $n$ is a multiple of an element in $\xi(f)$. The set $\xi(f)$ has first been introduced in \cite{daemen1995cipher}. As an example, $\xi(\chi)=\{2\}$, because $\chi$ is invertible if and only if $n$ is odd. There are very few functions $f$ for which $\inv(f)$ and $\xi(f)$ are known. Even the question of whether they are finite is in general difficult to answer. Using our results on the correspondence between permutations in $G_n$ and units in the residue ring $\F_2[X]/(X^{(n+1)/2})$, respectively $\F_2[X]/(X^n+X^{n/2})$ depending on parity of $n$, we determine $\inv(f)$ and $\xi(f)$ for every $f\in G$.

For a polynomial $F(X) \in \F_2[X]$ with $F(0) = 1$, the smallest
natural number $\ell$ such that $F(X)$ divides $X^\ell +1$ is called the order of $F(X)$. We denote it by $\ord(F)$. If $F(X) \in \F_2[X]$ is irreducible of degree $d\geq 2$, then $\ord(F)$ is a divisor of $2^d-1$. Hence $\ord(F)$ is odd.
Moreover, $F(X)$ divides $X^m+1$ if and only if $\ord(F)$ divides $m$. For more information on the order of a polynomial, see for example \cite{Lidl:1997}.

\begin{theorem}\label{thm:determination_of_xi}
    Let $f\in G$ with $\varphi(f)=[F(X)]$ and $F=g_1^{e_1}\cdots g_t^{e_t}\in\F_2[X], e_i\geq 1$ be its factorization into irreducible factors. Then
    \begin{equation*}
        \xi(f) = \{2\ord(g_i): i\in\{1,\ldots,t\}\}.
    \end{equation*}
    In particular, $\xi(f)$ is finite and $\inv(f)$ is infinite.
\end{theorem}

\begin{proof}
    If $n$ is odd, then $f$ is a permutation by Theorem 3 in \cite{kriepke2024algebraic}.

    If $n=2^sm$ is even with $m$ odd and $s\geq 1$, then by \cref{thm:permutation_iff_gcd_1} $f$ is a permutation if and only if $\gcd(F(X), X^m+1)=1$. This is equivalent to $\gcd(g_i(X),X^m+1)=1$ for all $i=1,\ldots ,t$. As $g_i$ is irreducible, this is equivalent to $g_i$ not being a divisor of $X^m+1$ which is equivalent to $\ord(g_i)\nmid m$. Hence, $f$ is a permutation for an even $n$ if and only if $n$ is not a multiple of $\ord(g_i)$ for every $i=1,\ldots, t$.
\end{proof}

\begin{remark}
    \cref{thm:determination_of_xi} and its proof imply that for any set  $T=\{2t_1, 2t_2,\ldots, 2t_u\}$ with odd $ t_i, 1\leq i \leq u$,  there exists a shift-invariant function $f$ such that $\xi(f)=T$. Indeed, for a given odd $t_i$ there is an irreducible polynomial $g_i\in\F_2[X]$ with $\ord(g_i)=t_i$. Such a polynomial exists: $t_i\mid 2^{m_i}-1$ for some $m_i \geq 1$, ensuring the existence of an element $\alpha_i\in\F_{2^{m_i}}$ with $\ord(\alpha_i)=t_i$ and thus the minimal polynomial of $\alpha$ has the required properties for $g_i$. For $F=g_1\cdots g_u$ and $f=\psi([F])$, it holds $\xi(f)=T$. 
\end{remark}

 The following result is a direct consequence of \cref{thm:determination_of_xi}. It could be useful in applications, since it allows to avoid an explicit computation of the order of the involved irreducible factors.
\begin{corollary}\label{cor:xi_is_subset}
    For $F\in\F_2[X]$ let $F=g_1^{e_1}\cdots g_t^{e_t}, e_i \geq 1,$ be its factorization into irreducible factors and set $d_i=\deg(g_i)$. Define $f=\psi([F])$. Then
    \begin{equation*}
        \xi(f) \subseteq \{2\ell: \ell\mid 2^{d_i}-1 \text{ for some } 1\leq i \leq t\}.
    \end{equation*}
\end{corollary}

Another  immediate consequence of \cref{thm:permutation_iff_gcd_1} is that for $f \in G$,
with the exception of some special type of multiples, the converse of \cref{lem:inv_closed_divisors} holds too. More precisely,
with the notation of \cref{thm:determination_of_xi},
if $f:\F_2^n \to \F_2^n$ is a permutation, then 
$f:\F_2^{nv} \to \F_2^{nv}$ is a permutation as well for any $v\geq 1$ for which
 $nv$ is not divisible by $2\ord(g_i)$
for every $1\leq i \leq t$. In particular,
when $v$ is a power of $2$ we get:

\begin{corollary}\label{cor:n_reducesto_m}
    Let $n=2^s\cdot m$ with $s\geq 2$ and $m\geq 1$ odd.  Then $f \in G$ is a permutation on $\F_2^n$ if and only if 
    it is a permutation on $\F_2^{2m}$.
\end{corollary}

\medskip

An interesting feature of permutations from $G$ is
that they permit  additive manipulations (as we describe below). This is a rather rare behavior for a family
of permutations. Even for linear functions,
it is difficult to assure that additive manipulations
lead to bijections. Another family of permutations, which 
tolerates  special additive changes is the one constructed via 
change of coordinate functions with a linear structure
\cite{kim2017groups, evoyan2013k, kyureghyan2011constructing, qin2015new}.\\

A function $f \in \Gamma$ is a permutation on $\F_2^n$ if and only if the polynomial $F \in \F_2[X]$, corresponding to it
via $\varphi(f) = [F]$ in $\mathcal{M}_n$, is co-prime with
$X(X^m+1)$. The sum of $F$ with any multiple of $X(X^m+1)$
remains clearly co-prime with $X(X^m+1)$. This simple 
observation allows to generate from a given permutation
$f$ on $\F_2^n$ numerous further permutation on
$\F_2^{2^sn}$ for any $s\geq 0$. 
We demonstrate this with an example.

\begin{example}
   The function $\tau = \gamma_0 + \gamma_2+\gamma_6$ has algebraic degree $4$ and is bijective on $\F_2^{22}$, 
   since $\varphi(\tau) = 1+X+X^3$ is irreducible and by \cref{cor:xi_is_subset} the set $\xi(\tau) = \{2\cdot(2^3-1)=14\}$ (that is, $\tau$ is a permutation on $\F_2^n$ for every $n$ not divisible by $14$).
   In particular, $\varphi(\tau) = 1+X+X^3$ is co-prime with $1+X^{11}$.  \cref{cor:n_reducesto_m} yields that $\tau$ is a permutation on $\F_2^n$ for every   $n=2^s\cdot 11$ with $s\geq 1$. 
   Next we generate further  permutations on  $\F_2^{2^s\cdot 11}$ using $\tau$, equivalently using the polynomial $1+X+X^3$. Note, that
   the polynomial $$H(X) =1+X+X^3 + G(X)X(X^{11}+1)$$ is co-prime with $X^{11}+1$ for any $G\in \F_2[X]$. Hence the function 
   $\tau ' = \psi([H])$ associated to it is a permutation on $\F_2^{2^s\cdot 11}$ for every $s\geq 1$. With the concrete choice
   $G(X)=1$ and thus $H(X) = 1+X+X^3 + X(X^{11}+1)$ we get
   that $\psi(1 +X^3 + X^{12} ) = \gamma_0 + \gamma_6+\gamma_{24} =\tau '$  is a permutation on $\F_2^{2^s\cdot 11}$ for every $s\geq 2$.
   The function $\tau '$ differs from $\tau$ and it has algebraic degree 12 if $s=1$ and $13$ otherwise by \cref{lem:gamma_2k_algebraic_degree}. 
\end{example}

\section{The properties of \texorpdfstring{$\kappa = \gamma_0+\gamma_2+\gamma_4$ in $G_n$}{γ0+γ2+γ4}}\label{sect:1+x+x^2_permutation}

To conclude this paper, we apply the results of the previous sections to study the permutations corresponding to   $1+X+X^2 \in \F_2[X]$.
This polynomial is  the simplest one which for an even $n$ may define a unit in $\M_n$, while $1+X$ is so for an odd $n$. Recall that $1+X$ corresponds to $\chi$.
We denote $\psi([1+X+X^2]) =\gamma_0+\gamma_2+\gamma_4 $ by $\kappa$. \\

\begin{lemma}\label{lem:1+x+x^2_permutation}
    Let $n\geq 1$. The function $\kappa=\gamma_0+\gamma_2+\gamma_4$ is a permutation on $\F_2^n$ if and only if $n$ is not a multiple of $6$.
\end{lemma}

\begin{proof} 
    We apply \cref{thm:determination_of_xi}. We have $F(X)=1+X+X^2$ which is irreducible with order $\ord(F)=3$, since $(1+X)(1+X+X^2)=1+X^3$. Hence $\xi(\kappa)=\{6\}$.
\end{proof}

A typical way to describe shift-invariant functions is by using the so-called complementing landscape (CL), see for instance \cite{toffoli1990invertible}. For example, $\chi$ is given by the CL $\chi = *01$. To explain what that means, let $y=\chi(x)$. Then the bit $x_i$ gets flipped by $\chi$ if and only if $x_i$ is followed by the pattern $01$, i.e. $(x_{i+1},x_{i+2})=(0,1)$. It can be quickly seen that this is equivalent to the previously mentioned formula $y_i = x_i + (1+x_{i+1})x_{i+2}$. 

Next we  determine the CL for $\kappa =\gamma_0+\gamma_2+\gamma_4$. Let $y=\kappa(x)$. Then
\begin{equation*}
    y_i = x_i + (1+x_{i+1})x_{i+2} + (1+x_{i+1})(1+x_{i+3})x_{i+4}.
\end{equation*}
Therefore, the bit $x_i$ gets flipped by $\kappa$ if and only if $(x_{i+1},x_{i+2},x_{i+3},x_{i+4})$ satisfies 
\begin{equation*}
    (1+x_{i+1})x_{i+2} + (1+x_{i+1})(1+x_{i+3})x_{i+4}=1.
\end{equation*}
The first term equals $1$ if and only if $(x_{i+1},x_{i+2})=(0,1)$ while the second term equals $1$ if and only if $(x_{i+1},x_{i+2},x_{i+3},x_{i+4})=(0,-,0,1)$, where the $-$ means that $x_{i+2}$ is arbitrary. If the first term is $1$, then the other term must be $0$, so if $(x_{i+1},x_{i+2})=(0,1)$, then either $x_{i+3}=1$ or $x_{i+4}=0$. If the second term is $1$ then the first term must be $0$, so if $(x_{i+1},x_{i+2},x_{i+3},x_{i+4})=(0,-,0,1)$ then $x_{i+2}=0$. This leads to the CL of $\kappa$
\begin{equation*}
     *011 \lor *01{-}0 \lor *0001.
\end{equation*}
This CL does appear in Table A.2 of Joan Daemen's thesis \cite{daemen1995cipher}. In \cite{daemen1995cipher} it is shown that $\kappa$ is a permutation if $n$ is odd and stated that it is invertible if and only if $n$ is not a multiple of $6$. The latter statement follows also from \cref{lem:1+x+x^2_permutation}. \\

Next we apply the correspondence  between permutations in $G_n$ and the units in $\M_n$ to
 determine the inverse function of $\kappa$. For that we describe the inverse of $[1+X+X^2]$ in $\F_2[X]/(m(X))$, where
\begin{equation*}
    m(X) = \begin{cases}
        X^{(n+1)/2} &  \text{if $n$ is odd} \\
        X^n+X^{n/2} &  \text{if $n$ is even.}
    \end{cases}
\end{equation*}
Recall that for any $k\geq 1$ it holds
\begin{equation*}
    1+X^{3k} = (1+X^3)(1+X^3+\ldots+X^{3(k-1)}) = (1+X+X^2)(1+X)(1+X^3+\ldots+X^{3(k-1)}).
\end{equation*}
We denote by $P_0(x) =0 $ and 
for  $k \geq 1$ by $P_{3k}(X)$  the polynomial
$$
 P_{3k}(X) := \frac{1+X^{3k}}{1+X+X^2}.
$$
Note, that
\begin{equation*}
    P_{3k}(X) = (1+X)(1+X^3+\ldots+X^{3(k-1)} ) = \sum_{i=0}^{3(k-1)+1} a_i X^i,
\end{equation*}
where $a_i=0$ if and only if $i=2 \bmod 3$. Observe that $P_{3k}(X)$ has degree $3k-2$. \\

The next lemma determines the inverse of $[1+X+X^2]$ for $n$ odd. In particular it shows that for $n=5$ we have $\kappa^{-1}=\chi$.
\begin{lemma}\label{lem:1+x+x^2_inverse_n_odd}
    Let $n\geq 1$ be odd and $k$ be minimal such that $3k\geq (n+1)/2$, that is 
    \begin{equation*}
        k = \begin{cases}
            \frac{n+5}{6}   & \text{if } n=1 \bmod 6 \\
            \frac{n+3}{6}   & \text{if } n=3 \bmod 6 \\
            \frac{n+1}{6}   & \text{if } n=5 \bmod 6.
        \end{cases}
    \end{equation*}
    The inverse of $[1+X+X^2]$ in $\F_2[X]/(X^{(n+1)/2})$ is 
    \begin{equation*}
        [1+X+X^2]^{-1} = [P_{3k}(X)] = 
            \begin{cases}
                [P_{3k}(X)] & 
                \text{if } n = 3,5 \bmod 6\\
                [P_{3k}(X)-X^{3k-2}] & 
                 \text{if } n = 1\bmod 6.
            \end{cases}
    \end{equation*}
\end{lemma}

\begin{proof}
    Note that
    \begin{equation*}
        \left[1+X+X^2\right]\left[P_{3k}(X)\right]=\left[X^{3k}+1\right]=\left[1\right],
    \end{equation*}
    since  $3k\geq (n+1)/2$.
It remains to observe that the degree of $P_{3k}$ is less than
$(n+1)/2$ if $n = 3,5 \bmod 6$. If $n=1 \bmod 6$, then $\deg P_{3k}(X)=3k-2 = 3\cdot\frac{n+5}{6}-2 = \frac{n+1}{2}$, so that the highest degree term  in $[P_{3k}(X)]$ vanishes modulo $X^{(n+1)/2}$.
 \end{proof}

The inverse of $[1+X+X^2]$ for $n$ even is as follows:

\begin{lemma}\label{lem:1+x+x^2_inverse_n_even}
    Let $n\geq 2$ be even and not a multiple of $6$. The inverse of $[1+X+X^2]\in\F_2[X]/(X^n+X^{n/2})$ is 
    \begin{equation*}
        [1+X+X^2]^{-1} = 
            \begin{cases}
                [1+XP_{3k}(X)+X^2P_{6k}(X)] & 
                \text{if } n=6k+2 \\
                [1+X^2P_{3k}(X)+XP_{6k+3}(X)] & 
                \text{if } n=6k+4.
            \end{cases}
    \end{equation*}
\end{lemma}

\begin{proof}
    We consider only $n=6k+2$ for some $k\in\N$, since the other case is analogous. We have 
    \begin{align*}
        1+X^n+X^{n/2} &= 1+X+X^2 + X^{6k+2}+X^2 + X^{3k+1}+X \\
            &= 1+X+X^2 + X^2 (X^{6k}+1) + X(X^{3k}+1) \\
            &= 1+X+X^2 + X^2(1+X+X^2)P_{6k}(X) +X(1+X+X^2)P_{3k}(X) \\
            &= (1+X+X^2)(1+X^2P_{6k}(X)+XP_{3k}(X))
    \end{align*}
    from which the claim follows. 
\end{proof}

If $n$ is even, the sequence of the coefficients $a_i$ of $[1+X+X^2]^{-1}=[\sum_{i=0}^k a_i X^i]$ has an interesting pattern as can be seen from \cref{tab:pattern_for_inverse}.  The coefficients yield a palindrome starting with a repeating pattern of $110$, ending with a repeating pattern of $011$ and meeting in the middle with either $01110$ or $010$, depending on $n=1,2 \mod 3$. \\

\begin{table}
    \centering 
    \begin{tabular}{c|l}
        $n$ & coefficients of $[1+X+X^2]^{-1}$ \\
        \hline 
        $8$ & 1101011 \\
        $10$ & 110111011 \\
        $14$ & 1101101011011 \\
        $16$ & 110110111011011 
    \end{tabular}
    \caption{Coefficients of $[1+X+X^2]^{-1}$ for $n$ even. As an example, for $n=8$ the inverse is given by $[1+X+0X^2+X^3+0X^4+X^5+X^6]$.\label{tab:pattern_for_inverse}}
\end{table}

We summarize our observations on $\kappa$ the following theorem.
\begin{theorem}
    Let $n\geq 4$. 
    The map $\kappa =\gamma_0+\gamma_2+\gamma_4$ is a permutation on $\F_2^n$  if and only if $n$ is not a multiple of $6$. The inverse $\kappa^{-1}$ is given by $\psi([1+X+X^2]^{-1})$ described in \cref{lem:1+x+x^2_inverse_n_odd,lem:1+x+x^2_inverse_n_even}. The algebraic degree of $\kappa$ is $3$ and the algebraic degree $d$ of $\kappa^{-1}$ is 
    \begin{equation*}
        d = \begin{cases}
            n/2 +1      & \text{if } n=0,2,4 \bmod 6 \\
            (n+1)/2     & \text{if } n=1,3 \bmod 6 \\
            (n-1)/2     & \text{if } n=5 \bmod 6 
        \end{cases}
    \end{equation*}
\end{theorem}

\begin{proof} 
    The statement follows from \cref{lem:1+x+x^2_permutation,lem:1+x+x^2_inverse_n_odd,lem:1+x+x^2_inverse_n_even}. The algebraic degree of $\gamma_{2k}$ is given in \cref{lem:gamma_2k_algebraic_degree} for $n$ even and in \cite[Lemma 2]{kriepke2024algebraic} for $n$ odd. 
\end{proof}

Our numerical results show that the cryptographic properties of $\kappa$ are comparable with those of $\chi$.
In particular, our computations for small values of $n$ yield the following open question.

\begin{problem}
For any  $n\geq 5$ the differential uniformity of $\kappa$ on $\F_2^n$ is $2^{n-2} - 2^{n-5}$.
\end{problem} 

The differential uniformity of $\chi$ on $\F_2^n$ is $2^{n-2}$ which has been known for a long time \cite{daemen1995cipher}. However, to the best of our knowledge, all published proofs of this fact are recent, see \cite{schoone2024algebraic, graner2025bijectivity}.

\bibliographystyle{splncs04.bst}
\bibliography{refs.bib}

\end{document}